\documentclass[10pt,a4paper]{article}
\usepackage{latexsym,amsfonts,amsmath,amsthm,amssymb,mathbbol}
\usepackage{bbm}
\usepackage{wasysym}
\usepackage[pdftex]{graphicx}
\usepackage{cite}

\usepackage[english]{babel}
\usepackage{amsmath}
\usepackage{mathtools}
\usepackage[latin1]{inputenc}
\usepackage{fancybox}

\usepackage{times}
\usepackage[T1]{fontenc}
\usepackage{url}

\title{On an analytical method for the satellite problem revisited }
\author{
  \qquad
  Niccol\`o Pilloni and
  Claudio Saccon
  \footnote{Dipartimento di Matematica,
    Largo Bruno Pontecorvo 5, I56127 Pisa, ITALY, e-mail: claudio.saccon@unipi.it}
  }
\date{}

\newcommand{\real}{\ensuremath{\mathbb{R}}}

\newcommand{\secl}[1]{\ensuremath{\left[#1\right]_{sec}}}
\newcommand{\perd}[1]{\ensuremath{\left[#1\right]_{per}}}

\newcommand{\aaa}{a}
\newcommand{\incl}{\ensuremath{I}}
\newcommand{\rr}{\mathbf{r}}

\newcommand{\KK}{\mathcal{K}}
\newcommand{\HH}{\mathcal{H}}
\newcommand{\jj}{\mathcal{J}}

\newcommand{\X}[1][{}]{\ensuremath{\mathbb{X}}}

\newcommand{\norm}[1]{\ensuremath{\left\lVert #1 \right\rVert}}

\swapnumbers
\theoremstyle{plain}
\newtheorem{thm}{Theorem}[section]

\newtheorem{prop}[thm]{Proposition}

\theoremstyle{definition}
\newtheorem{dfn}[thm]{Definition}

\theoremstyle{remark}
\newtheorem{rmk}[thm]{Remark}

\newcommand{\FF}{\mathbf{F}}

\begin{document}

\maketitle

\section{Introduction}

In this note we discuss an \textbf{analitical} method, which goes back to  Dirk Brouwer (1959 see \cite{Brouwer1959}), to compute the trajectories of an artificial satellite rotating around 
the Earth when
the oblate shape  is taken into account. Despite being somewhat old 
these kind of techniques are still used
nowadays in software packages  like SGP4 (see \cite{SGP4}) used for the tracking of space debris. Actually the models on which these softwares are based
are the result of later generalizations 
(see \cite{BrouwerHori1961,Lane1965,LaneCrawford1969}) which include 
the effect of the atmosferic drag. We don't cover the general theory
here -- we just remark that the problem with drag is treated as a pertubation of the dragless case.

As well known the Hamiltonian associated with the ``non spherical Earth'' can 
be written as a series of ``harmonics'' of three types:
zonal (only depending on the latitude), sectorial (only depending on the longitude), and tesseral (depending on both) each 
having a suitable coefficient. The most important term (apart
from the zero-th one which corresponds
to the perfectly spherical case)
is the coefficient $J_2$ of the first zonal harmonic.
The main idea in \cite{Brouwer1959} is looking for a canonical tranformation to a new set of variables such that the new Hamiltonian only depends on the momenta. If this is possible
then the new Hamilton equations become trivial: the momenta are constant while the coordinates are linear (this goes with a price: 
the transformation is usually complicated and made up of a very long list of terms). 
To accomplish this a Brouwer looks for a generating function of the 
form of a power series in the coefficient $J_2$. 
Using some calculus involving composition of power series he
derives some conditions on the coefficients of such a series so that the associated canonical tranformation has the 
desired property. In the second part of \cite{Brouwer1959} also
takes into account the coefficient $J_4$ of the second zonal
harmonic and in the last part some sectorial harmonic are considered.
It has to be said that Brouwer credits
Hugo von Zeipel in \cite{VonZeipel1916} for 
the above mentioned idea of the power series expansion. Unfortunately we were not able to find the original article so our main source is \cite{Brouwer1959}: nevetheless we use the term ``von Zeipel method'' to indicate the technique we want to discuss.

The porpose of the present work is to give a general and coincise 
presentation of the von Zeipel method
by using  a more up-to-date vectorial notation. We use a multi dimensional power series expansion, corresponding to a generic set of parameters 
(not just $J_2$). In this way we can derive a more transparent general formula, which allows to replicate
the computations of \cite{Brouwer1959} ``all at once''. Since such computation involve a lot of
derivatives, we have used Wolfram Mathematica to 
perform the actual work. In this way we could also compute a second order term for the generating function that \cite{Brouwer1959} omits. The
corresponding Mathematica notebook that we wrote for this is publicly avaliable at the link:
https://www.wolframcloud.com/obj/claudiosaccon/VonZeipel2.nb
\bigskip

In order to make this note more self--contained, we added an appendix where we recall the definitions and the main properties of canonical transformations and generating functions.
\bigskip

We remind that, if $f:\real^N\to\real^M$, the \emph{Jacobian matrix} of $f$ is defined by:
% 
% FORSE NO - NOTAZIONE DI EINSTEIN??

\[
 J_f(x)=\frac{\partial f}{\partial x}(x):=
 \begin{pmatrix}
  \dfrac{\partial f_1}{\partial x_1}(x)
  &\cdots
  &
  \dfrac{\partial f_1}{\partial x_N}(x)
  \\
  \vdots&\ddots&\vdots
  \\
  \dfrac{\partial f_M}{\partial x_1}(x)
  &\cdots
  &
  \dfrac{\partial f_M}{\partial x_N}(x)
 \end{pmatrix}.
\]
If $M=1$, i.e. if $f$ is a scalar function, we sometimes consider the \emph{gradient } of $f$:
\[
 \nabla f(x):=J_f(x)^t=
 \begin{pmatrix}
  \dfrac{\partial f}{\partial x_1}(x)
  \\
  \vdots
  \\
  \dfrac{\partial f}{\partial x_N}(x)
 \end{pmatrix}.
\]
When more than two indices are needed (for instance when
we have to deal with second order derivatives of vectors or matrices) we use 
Einstein's convention of repeated indices.
In this note elements of $\real^N$ are regarded as
``vectors'', i.e. as
$N$ by $1$ (column) matrices.

%   The last term in \eqref{eqn:problem-debris}
%   models the ``drag'', that is the resistance of the air acting
%   on the body. Notice that the drag term has the direction
%   of the velocity $\dot r$, with opposite verse. Moreover its
%   amplitude is proportional to the density of the athmosfere
%   $\rho(|r|)$ (which is assumed to only depend on the height $|r|$)
%   and to the square of the velocity. The term $B$ is a suitable
%   constant (which makes the equation dimensionally correct).
%   Concerning the model for the density, we consider:
%   \[
%    \rho(r)=\frac{1}{(r-r_0)^k}
%    \qquad
%    \mbox{for }r>r_0.
%   \]
%   As far as we understand such a power law
%   has just been chosen since it is easier to treat (XXXX). An exponential behavior could also be considered 
%   but it would lead to much more 
%   complicated computations involving  an additional
%   power series expansion (XXXX).
%   The parameters $\mathbf{r}_0\geq0$ and
%   $k$ integer are determined by trying to fit experimental data

\section{The satellite problem}

The satellite problem can be written as:
  \begin{equation}\label{eqn:problem-debris}
   \ddot\rr=-\frac{\mu\rr}{|\rr|^3}-\nabla U(\rr)+
   \FF(\rr)
  \end{equation}
  The term $\mu$ stands for $GM$ where $M$ is the
  mass of the Earth and $G$ is the Gravitational Constant.
  In  \eqref{eqn:problem-debris} $\nabla U$ is an additional conservative term taking into
  account the non spherical shape of the Earth. In the case of the ``oblated Earth'' 
  \[
   U(\rr)=
   \sum_{n=2}^\infty\frac{\mu}{|\rr|}
   J_n \left(\frac{R}{|\rr|}\right)^nP_n\left(\sin(\beta)\right)
  \] 
  In this formula $\beta$ denotes the latitude, $R$ is the radius oth the Earth, and $P_n$ are the Legrendre polinomials. In this
  description the potential is still rotationally symmetric
  (only ``zonal harmonics are considered). The even terms are also symmetric with respect to the equator, while 
  the odd ones introduce an asymmetry between the north and the
  south emisphere. In \cite{Brouwer1959} terms up to $J_5$  are 
  considered (with different names).  The methods presented in this note rely on the fact that the coefficient $J_n$ and are small.
  
  The term $\FF$ allows to take into account external, nonconservative forces acting on the satellite. The most important
  effect of this type is the resistance of the athmosphere, usually
  called \emph{drag}. 
  
  In this note we drop the term $\FF$. In 
  \cite{BrouwerHori1961,Lane1965,LaneCrawford1969}
  the problem with drag is considered (as a perturbation of the case without drag), but the
  technique we are going to present here 
  only works in the case $\FF=0$. In the Appendix, 
  however, we show how the full problem gets 
  transformed under a canonical tranformation, so 
  the reader can understand how this technique can be relevant in the problem with drag.
  \bigskip

\section{The Von Zeipel method}  
Now we take a more general point of view.
Let $\HH:\real^N\times\real^N\times\real^M\to\real$,
$\HH=\HH(p,q,\jj)=\HH(v,\jj)$ be a Hamiltonian depending on M parameters $J_1,\dots,J_M$ which  form the vector $\jj$ 
($J_i$ will be supposed to be small). We remind that for
an $M\times M$ matrix the expression $A\jj^2$ denotes
$\jj^t A\jj$.
We suppose that 
\[
 \HH(v,\jj)=
 \HH_0(v)+
 \HH_1(v)\cdot \jj+
 \HH_2(v)\jj^2+
 0(|k^2|).
\]
where $\HH_0:\real^{2N}\to\real$ (scalar), 
$\HH_1:\real^{2N}\to\real^{M}$ (vector valued), and 
$\HH_2:\real^{2N}\to\real^{M^2}$ (matrix valued).
Of course (by Taylor's expansion:
\[
 \HH_0(v)=\HH(v,0),\quad
 \HH_1(v)=\nabla_k\HH(v,0),\quad
 \HH_2(v)=2\nabla_k^2\HH(v,0)
\]
($\nabla_k$ and $\nabla^2_k$ denote the gradient and the Hessian
matrix with respect to the $k$ variable).

We assume that:
\begin{gather}\label{eqn:ipotesi-zeipel-1}
 \HH_0\mbox{ does not depend on $q$};
 \\
 \label{eqn:ipotesi-zeipel-2}
 \HH_1\mbox{ and }\HH_2\mbox{ are $2\pi$ periodic in each }q_i
 \quad i=1,\dots,N.
\end{gather}
The periodicity assumption means that the variables
$q_1,\dots,q_N$ are \emph{angles}. Given any function
$f=f(p,q)$ we can  define
the \emph{secular part } of $f$ as:
\[
 f_{sec}(p):=
 \frac{1}{(2\pi)^N}
 \int_0^{2\pi}\cdots\int_0^{2\pi}f(p,q)\,dq_1\dots dq_N
\]
(i.e the average in $q$ over the cube
$[0,2\pi]\times\cdots\times[0,2\pi]$) and the \emph{periodic part} of $f$:
\[
 f_{per}(p,q):=f(p,q)-f_{sec}(p).
\]
It is clear that $f\mapsto f_{sec}$ and $f\mapsto f_{per}$
are linear, and that  $\secl{f_{sec}}=f_{sec}$, 
$\perd{f_{per}}=f_{per}$, $\perd{f_{sec}}=0=\secl{f_{per}}$.
Moreover $f=f_{sec}$ if and only if $f$ does not depend
on $q$.

The purpose of this section is finding a canonical transformation
such that the new Hamiltonian $\KK$ only depends on the (new)
momenta $P$. To this aim 
we look for a generating function of the form:
\[
 S(P,q,\jj)=
 S_0(P,q)+
 S_1(P,q)\cdot\jj+
 S_2(P,q)\jj^2+o(|\jj|^2)
\]
(again with  $S_0$ is scalar, $S_1$ an $M$ vector, and $S_2$ 
an $M\times M$ matrix. For the moment 
we assume that such an $S$ is given and derive some 
relationships between the old  Hamiltonian $\HH$ and the new one 
$\KK$.
We then use such relationships to choose $S_0$, $S_1$, and $S_2$
in such a way that $\KK$ exhibis the desired property.

If $V(v)=\begin{pmatrix}P(p,q)\\Q(p,q)\end{pmatrix}$ be  the new variables induced by $S$, and $\KK=\KK(V)=\KK(P,Q)$ be the new 
Hamiltonian. We can write
\[
 \KK(V,k)=
 \KK_0(V)+
 \KK_1(V)\cdot\jj+
 \KK_2(V)\jj^2+
 0(|\jj^2|).
\]
where $\KK_0$, $\KK_1$, and $\KK_2$ have the same properties of
$\HH_0$, $\HH_1$, and $\HH_2$. 
We know that:
\[
 \KK\left(P(p,q,\jj),\frac{\partial S(P(p,q),q,\jj)}{\partial P},k\right)=
 \HH\left(\frac{\partial S(P(p,q),q,\jj)}{\partial q},q,\jj\right)
 \quad\forall p, q, \jj.
\]
In the following we are going to keep things coincise by skipping
some of  the ``inner'' variables. 
By Taylor's expansion:
\begin{multline*}
 \HH\left(\frac{\partial S}{\partial q},q,\jj\right)=
 \HH_0\left(\frac{\partial S}{\partial q},q\right)+
 \HH_1\left(\frac{\partial S}{\partial q},q\right)^t\jj+
 \HH_2\left(\frac{\partial S}{\partial q},q\right)\jj^2+
 0(|\jj^2|)=
 \\
 \HH_0%\left(\frac{\partial S_0}{\partial q},q\right)
 +
 \frac{\partial \HH_0}{\partial p_i}
 %\left(\frac{\partial S_0}{\partial q},q\right)
 \frac{\partial S_{1,k}}{\partial q_i}J_k+
 \frac{\partial \HH_0}{\partial p_i}
 %\left(\frac{\partial S_0}{\partial q},q\right)
 \frac{\partial S_{2,k,h}}{\partial q_i}J_kJ_h+
 %\\
 \frac{1}{2}
 \frac{\partial^2 \HH_0}{\partial p_i\partial p_j}
 %\left(\frac{\partial S_0}{\partial q},q\right)
 \frac{\partial S_{1,k}}{\partial q_i}
 \frac{\partial S_{1,h}}{\partial q_j}J_kJ_h+
 \\
 \HH_{1,k}
 %\left(\frac{\partial S_0}{\partial q},q\right)
 J_k+
 %\\
 \frac{\partial \HH_{1,h}}{\partial p_i}
 \left(\frac{\partial S_0}{\partial q},q\right)
 \frac{\partial S_{1,k}}{\partial q_i}J_kJ_h+
 \HH_{2,k,h}
 %\left(\frac{\partial S_0}{\partial q},q\right)
 J_kJ_h+
 o(|\jj|^2).
\end{multline*}
($\HH_i$ are computed in 
$\left(\frac{\partial S_0}{\partial q},q\right)$). 
In the same way:
\begin{multline*}
 \KK\left(P,\frac{\partial S}{\partial P},k\right)=
 \KK_0
 %\left(P,\frac{\partial S_0}{\partial P}\right)
 +
 \frac{\partial \KK_0}{\partial Q_i}
 %\left(P,\frac{\partial S_0}{\partial P}\right)
 \frac{\partial S_{1,k}}{\partial P_i}J_k+
 %\\
 \frac{\partial \KK_0}{\partial Q_i}
 %\left(P,\frac{\partial S_0}{\partial P}\right)
 \frac{\partial S_{2,k,h}}{\partial P_i}J_kJ_h+
 \\
 \frac{1}{2}
 \frac{\partial^2 \KK_0}{\partial Q_i\partial Q_j}
 %\left(P,\frac{\partial S_0}{\partial P}\right)
 \frac{\partial S_{1,k}}{\partial P_i}
 \frac{\partial S_{1,h}}{\partial P_j}J_kJ_h+
 %\\ 
 \KK_{1,k}
 %\left(P,\frac{\partial S_0}{\partial P}\right)^t
 J_k
 +
 \frac{\partial \KK_1}{\partial Q_i}
 %\left(P,\frac{\partial S_0}{\partial P}\right)
 \frac{\partial S_{1,k}}{\partial P_i}
 J_kJ_h+
 \KK_{2,k,h}
 %\left(P,\frac{\partial S_0}{\partial P}\right)
 J_kJ_h+
 o(|\jj|^2)
\end{multline*}
($\KK_i$ are computed in $\left(P,\frac{\partial S_0}{\partial P}\right)$).
% \begin{multline*}
%  \KK\left(P,\frac{\partial S}{\partial P},k\right)=
%  \KK_0\left(P,\frac{\partial S}{\partial P}\right)+
%  \KK_1\left(P,\frac{\partial S}{\partial P}\right)\cdot k+
%  \KK_2\left(P,\frac{\partial S}{\partial P}\right)k^2+
%  0(|k^2|)=
%  \\
%  \KK_0\left(P,\frac{\partial S_0}{\partial P}\right)+
%  \frac{\partial \KK_0}{\partial Q}
%  \left(P,\frac{\partial S_0}{\partial P}\right)
%  \left(\frac{\partial S_1}{\partial P}\cdot k\right)+
%  \frac{1}{2}
%  \frac{\partial^2 \KK_0}{\partial Q^2}
%  \left(P,\frac{\partial S_0}{\partial P}\right)
%  \left(\frac{\partial S_1}{\partial P}\cdot k\right)^2+
%  \\
%  \frac{\partial \KK_0}{\partial Q}
%  \left(P,\frac{\partial S_0}{\partial P}\right)
%  \frac{\partial S_2}{\partial P}\cdot k^2+
%  \KK_1\left(P,\frac{\partial S_0}{\partial P}\right)\cdot k+
%  \frac{\partial \KK_1}{\partial Q}
%  \left(P,\frac{\partial S_0}{\partial P}\right)
%  \left(\frac{\partial S_1}{\partial P}\cdot k, k \right)+
%  \\
%  \KK_2\left(P,\frac{\partial S_0}{\partial P}\right)k^2+
%  o(|k|^2).
% \end{multline*}
%
%
%
%
% 
We take:
\[
 S_0(P,q)=P\cdot q
\]
(if no other terms were present, then the associated canonical
transformation would be the identity) and we try to find
$S_1$ and $S_2$ such that $\KK(V,k)=\KK(P,k)$, that is $\KK$
does not depend on $Q$. By equating the terms having
corresponding powers of $\jj$, we get:
\begin{gather}\label{condition-A}
 \KK_0(P,Q)=\KK_0(P)=\HH_0(P)
 \\
 \label{condition-B}
 \frac{\partial \KK_0}{\partial Q_i}
 \frac{\partial S_{1,k}}{\partial P_i}+
 \KK_{1,k}=
 \frac{\partial \HH_0}{\partial p_i}
 \frac{\partial S_{1,k}}{\partial q_i}+
 \HH_{1,k}
 %\qquad k=1,\dots,M
 \\
 \notag
 \frac{\partial \KK_0}{\partial Q_i}
 \frac{\partial S_{2,k,h}}{\partial P_i}+
 \frac{1}{2}
 \frac{\partial^2 \KK_0}{\partial Q_i\partial Q_j}
 \frac{\partial S_{1,k}}{\partial P_i}
 \frac{\partial S_{1,h}}{\partial P_j}+
 \frac{\partial S_{1,k}}{\partial P_i}
 \frac{\partial \KK_{1,h}}{\partial Q_i}+
 \KK_{2,k,h}
 \\\label{condition-C}
 =
 \\
 \notag
 \frac{\partial \HH_0}{\partial p_i}
 \frac{\partial S_{2,k,h}}{\partial q_i}+
 \frac{1}{2}
 \frac{\partial^2 \HH_0}{\partial p_i\partial p_j}
 \frac{\partial S_{1,k}}{\partial q_i}
 \frac{\partial S_{1,h}}{\partial q_j}+ 
 \frac{\partial S_{1,k}}{\partial q_i}
 \frac{\partial \HH_{1,k}}{\partial p_i}+
 \HH_{2,k,h}
\end{gather}
(with $k$ and $h$ ranging between $1$ and $M$).
%We \textbf{assume} that $\HH_0(p,q)=\HH_0(p)$. 
Using  
\eqref{condition-A} we have found $\KK_0$ (which only depends
on $P$). 
Then all derivatives of $\KK_0$ in \eqref{condition-B} and \eqref{condition-C} are zero. 
If $\KK_1$ and $\KK_2$ are to be undependent of $Q$, then both \eqref{condition-B} and 
\eqref{condition-C}
split into  a pair of conditions:
\begin{gather}
\label{condition-B-S}
 \KK_{1,k}(P)=
 \secl{
 \frac{\partial S_{1,k}}{\partial q_i}
 \frac{\partial \HH_0}{\partial p_i}+
 \HH_{1,k}
 }
 \\
 \label{condition-B-P}
 0=
 \perd{ \frac{\partial S_{1,k}}{\partial q_i}
 \frac{\partial \HH_0}{\partial p_i}+
 \HH_{1,k}
 }
 \\
 \label{condition-C-S}
 \KK_{2,k,h}(P)=
 \secl{
 \frac{\partial \HH_0}{\partial p_i}
 \frac{\partial S_{2,k,h}}{\partial q_i}+
 \frac{1}{2}
 \frac{\partial^2 \HH_0}{\partial p_i\partial p_j}
 \frac{\partial S_{1,k}}{\partial q_i}
 \frac{\partial S_{1,h}}{\partial q_j}+ 
 \frac{\partial S_{1,k}}{\partial q_i}
 \frac{\partial \HH_{1,h}}{\partial p_i}+
 \HH_{2,k,h}
 }
 \\
 \label{condition-C-P}
 0=
 \perd{
 \frac{\partial \HH_0}{\partial p_i}
 \frac{\partial S_{2,k,h}}{\partial q_i}+
 \frac{1}{2}
 \frac{\partial^2 \HH_0}{\partial p_i\partial p_j}
 \frac{\partial S_{1,k}}{\partial q_i}
 \frac{\partial S_{1,h}}{\partial q_j}+ 
 \frac{\partial S_{1,k}}{\partial q_i}
 \frac{\partial \HH_{1,h}}{\partial p_i}+
 \HH_{2,k,h}
 }
\end{gather}
A possible way to verify \eqref{condition-B-S} 
and  \eqref{condition-B-P} is by imposing:
\begin{gather}
 \KK_{1,k}(P):=\secl{\HH_{1,k}(P)}\label{condition-B-S1}
 \\
 \frac{\partial \HH_0}{\partial p_i}(P)
 \frac{\partial S_{1,k}}{\partial q_i}(P,q)=
 -\perd{\HH_{1,k}(P,q)}
 \label{condition-B-P1}
\end{gather}
% (so actually the term inside the periodic part
% in \eqref{condition-B-P} is required to be zero).
In order to solve equation \eqref{condition-B-P1}:
let:
\begin{gather*}
 w(P):=\frac{\partial \HH_0}{\partial p}(P),
 \quad,\quad
 \hat w(P):=
 \norm{\dfrac{\partial \HH_0}{\partial p}(P)}^{-1}
 \dfrac{\partial \HH_0}{\partial p}(P)=
 \frac{w(P)}{\|w(P)\|}
\end{gather*}
(notice that the ``coefficient'' $w$ neither depends on $q$ nor on $k$) and define $S_1$ by:
\begin{multline*}
 S_{1,k}(P,q):=
 -\frac{1}{\|w(P)\|}
 \int_0^{q\cdot \hat w}
 \perd{
 \HH_{1,k}
 \left(
 P,
 q-(1-t)
 (\hat w(P)\cdot q)
 \hat w(P)
 \right)
 }
 \,dt=
 \\
 \frac{\secl{\HH_{1,k}}(P)}{\norm{w(P)}}\hat w\cdot q
 -\frac{1}{\norm{w(P)}}
 \int_0^{q\cdot \hat w}
 \HH_{1,k}
 \left(
 P,
 q-(1-t)
 (\hat w(P)\cdot q)
 \hat w(P)
 \right)
 \,dt.
\end{multline*}
It is then simple to check that $S_1$
 solves Equation \eqref{condition-B-P1}.
So we can find $\KK_1$ and $S_1$.
We can use the same idea to find $\KK_2$ and $S_2$: 
% If
% $\KK_2(P,q)=\KK_2(P)$, then 
% \eqref{condition-C} is equivalent to:  
% \begin{gather}
%  \KK_2(P)=
%  \label{condition-C-S-1}
%  \left[\frac{\partial \HH_0}{\partial p}
%  \frac{\partial S_2}{\partial q}+
%  \frac{1}{2}
%  \left(\frac{\partial S_1}{\partial q}\right)
%  \frac{\partial^2 \HH_0}{\partial p^2}
%  \left(\frac{\partial S_1}{\partial q}\right)^t+ 
%  \frac{\partial S_1}{\partial q}
%  \frac{\partial \HH_1}{\partial p}+
%  \HH_2\right]_s
%  \\
%  \label{condition-C-P-1}
%  0=
%  \left[\frac{\partial \HH_0}{\partial p}
%  \frac{\partial S_2}{\partial q}+
%  \frac{1}{2}
%  \left(\frac{\partial S_1}{\partial q}\right)
%  \frac{\partial^2 \HH_0}{\partial p^2}
%  \left(\frac{\partial S_1}{\partial q}\right)^t+ 
%  \frac{\partial S_1}{\partial q}
%  \frac{\partial \HH_1}{\partial p}+
%  \HH_2\right]_p
% \end{gather}

% \begin{gather}
%  \KK_2(P,q)=\KK_2(P)=
%  \HH_{2,s}(P)
%  \label{condition-C-S}
%  \\
%  \frac{\partial \HH_0}{\partial p}
%  \frac{\partial S_2}{\partial q}=
%   -\frac{1}{2}
%   \frac{\partial^2 \HH_0}{\partial p^2}
%   \left(\frac{\partial S_1}{\partial q}\right)^2-
%    \frac{\partial \HH_1}{\partial p}
%    \frac{\partial S_1}{\partial q}-
%    \HH_{2,p}(P,q)
%    \label{condition-C-P}
% \end{gather}
% (we have used the fact that  the derivatives of $\KK_0$ and $\KK_1$
% with respect to $Q$ are zero); notice that $S_1$  is known from
% the previous step. Using \eqref{condition-C-S} and \eqref{condition-C-P} we can find $\KK_2$ and $S_2$.
%

%following \cite{Brouwer1959}, we can choose:
%Concerning \eqref{condition-C-S-2} and \eqref{condition-C-P-2},

\begin{gather}
  \label{condition-C-S-2}
 \KK_{2,k,h}(P)=
 \secl{
%  \frac{\partial \HH_0}{\partial p_i}
%  \frac{\partial S_{2,k,h}}{\partial q_i}+
 \frac{1}{2}
 \frac{\partial^2 \HH_0}{\partial p_i\partial p_j}
 \frac{\partial S_{1,k}}{\partial q_i}
 \frac{\partial S_{1,h}}{\partial q_j}+ 
 \frac{\partial S_{1,k}}{\partial q_i}
 \frac{\partial \HH_{1,h}}{\partial p_i}+
 \HH_{2,k,h}
 }
 \\
 \label{condition-C-P-2}
 \frac{\partial \HH_0}{\partial p_i}
 \frac{\partial S_{2,k,h}}{\partial q_i}=
 -
 \perd{
 \frac{1}{2}
 \frac{\partial^2 \HH_0}{\partial p_i\partial p_j}
 \frac{\partial S_{1,k}}{\partial q_i}
 \frac{\partial S_{1,h}}{\partial q_j}+ 
 \frac{\partial S_{1,k}}{\partial q_i}
 \frac{\partial \HH_{1,h}}{\partial p_i}+
 \HH_{2,k,h}
 }
\end{gather}

Again \eqref{condition-C-S-2} provides the expression of $\KK_2$
(undependent of $Q$), while \eqref{condition-C-P-2}
can be used to find $S_{2,k,h}$ (although in \cite{Brouwer1959} $S_2$
is neglected ????).

\section{The solution of the satellite problem without drag}

Let $r:=|\rr|$, $\aaa$ denote the osculating semi-major axis,
$e$ denote the eccentricity, and $\nu$ denote the true 
anomaly.

We consider the \emph{Delaunay variables}:
\begin{align*}
 (p_1=)\quad L&:=(\mu\aaa)^{1/2},
 &
 (q_1=)\quad l&:=\mbox{ mean anomaly},
 \\
 (p_2=)\quad G&:=L(1-e^2)^{1/2},
 &
 (q_2=)\quad g&:=\mbox{ argument of the pericenter},
 \\
 (p_3=)\quad H&:=G\cos(\incl),
 &
 (q_3=)\quad h&:=\mbox{ longitude of ascending node}.
\end{align*}

We first consider the problem with $J_n=0$ if $n\neq2$ (only $J_2$ is present).
In terms of $(L,G,H,l,g,h)$ the Hamiltonian takes the form:
\[
 \HH(L,G,H,l,g,h):=
 \HH_0(L)+\HH_1(L,G,H,l,g)J_2
\]
where:
\begin{align}
 %\HH_0(L,G,H)=
 \HH_0(L)=&
 \frac{\mu^2}{2L^2}
 \\
 %HH_1(L,G,H,l,g,h)=
 \HH_1(L,G,H,l,g)=&
 \frac{\mu^4R^2}{4L^6G^2}
 \frac{\aaa^3}{\rr^3}
   \left[
     \left(
       -G^2+3H^2
     \right)
     +
     \right. 
     \\
     &\qquad+
     \left.
     3
     \left(
       G^2-H^2
     \right)
     \cos(2g+2\nu)
   \right]
   \notag
\end{align}

Notice that \eqref{eqn:ipotesi-zeipel-1} and \eqref{eqn:ipotesi-zeipel-2} hold. To maintain the
notations of \cite{Brouwer1959} we use doubly
primes to   denote the new variables:
$P_1=L''$,  $P_2=G''$, $P_2=H''$, $Q_1=l''$, $Q_2=g''$, $Q_3=h''$ (notice that we are skipping the single prime 
variables  since we are doing the two steps at once). So $\KK=\KK(L'',G'',H'')$ and 
$S=S(L'',G'',H'',l,g,h)$.

Of course the term $\dfrac{\aaa^3}{\rr^3}$ and the 
variable $\nu$ need to be expressed in terms of the
Delaunay variables. For this we can use the Fourier series expansions:
\begin{align}
 \label{eqn:series-espansion-1}
 \frac{\aaa^3}{\rr^3}&=
 \frac{L^3}{G^3}+
 \sum_{n=1}^\infty 2P_n(e)\cos(nl)
 \\
 \label{eqn:series-espansion-2}
 \frac{\aaa^3}{\rr^3}\cos(2g+2\nu)&=
 \sum_{n=-\infty}^\infty Q_n(e)\cos(2g+nl)
\end{align}
where $P_n$ and $Q_n$ are suitable \emph{Hansen coefficients}
% 
% VEDI ???

Notice that $e=\left(1-\dfrac{G^2}{L^2}\right)^{1/2}$
so  $P_n=P_n(L,G)$, $Q_n=Q_n(L,G)$.

From \eqref{eqn:series-espansion-1} and \eqref{eqn:series-espansion-2} we infer that:
\begin{align*}
 \HH_{1,sec}(L,G,H)=&
 \frac{\mu^4R^2}{4L^3G^5}
 \left(
       -G^2+3H^2
     \right)
 \\
 \HH_{1,per}(L,G,H)=&
  \frac{\mu^4R^2}{4L^6G^2}
   \left[
     \left(
       -G^2+3H^2
     \right)
     \left(
     \frac{\aaa^3}{\rr^3}-
     \frac{L^3}{G^3}
     \right)
     +
     \right. 
     \\
     &\qquad +\left.
     \left(
       3G^2-3H^2
     \right)\frac{\aaa^3}{\rr^3}\cos(2g+2\nu)
   \right]
\end{align*}

Using the above formula we get from conditions \eqref{condition-B-S1} 

\begin{equation}\label{eqn:expression-of-K1-I}
 \KK_1(L'',G'',H'')=
 \frac{\mu^4R^2}{4{L''}^3{G''}^5}
 \left(
   -{G''}^2+3{H''}^2
 \right)
\end{equation}
while
\eqref{condition-B-P1} turns into:
\begin{multline}\label{eqn:condition-on-S1-I}
 \frac{\partial S_1}{\partial l}(L'',G'',H'',l,g)=
 \frac{\mu^2R^2}{4{L''}^3{G''}^2}\times
 \\
 \times
   \left[
     \left(
       -{G''}^2+3{H''}^2
     \right)
     \left(
       \frac{\aaa^3}{\rr^3}-\frac{{L''}^3}{{G''}^3}
     \right)
     +
     3
     \left(
       {G''}^2-{H''}^2
     \right)\frac{\aaa^3}{\rr^3}\cos(2g+2\nu)
   \right].
\end{multline}
If we use the Fuorier expansions mentioned above, we get:
\begin{multline*}
 S_1(L'',G'',H'',l,g)=
 \frac{\mu^2R^2}{4{L''}^3{G''}^2}
  \left[
     \left(
       3{H''}^2-{G''}^2
     \right)
     \sum_{n=1}^\infty \frac{2P_n}{n}(L'',G'')\sin(nl)
  \right.
    +
  \\
  \left.
     3
     \left(
       {G''}^2-{H''}^2
     \right)
     \sum_{n=-\infty}^\infty \frac{Q_n}{n}(L'',G'')\sin(2g+nl)
   \right]
\end{multline*}
(notice that $S_1$ does not depend on $h$).
However, using the formulas:
\begin{equation}\label{eqn:change-of-variable-l-nu}
 dl=\frac{L}{G}\frac{r^2}{\aaa^2}d\nu
 \quad,\quad
 \frac{\aaa}{r}=
 \frac{1+e\cos(\nu)}{1-e^2}=
 \frac{L^2}{G^2}\left(1+\sqrt{1-\frac{G^2}{L^2}}\cos(\nu)\right)
\end{equation}
we can also find a closed formula for $S_1$. From \eqref{eqn:condition-on-S1-I} we get:
\begin{multline}
 S_1(L'',G'',h'',l,g)=
 \frac{\mu^2R^2}{4{G''}^5}({G''}-3{H''}^2)l+
 \\
 \frac{\mu^2R^2}{4{L''}^3{G''}^2}
 \int\left[
     \left(
       -{G''}^2+3{H''}^2
     \right)
     \left(\frac{L}{G}\right)
     \frac{1+e\cos(\nu)}{1-e^2}
     +
     \right.
     \\
     \left.
     3\left(
       {G''}^2-{H''}^2
     \right)
     \left(\frac{L}{G}\right)
     \frac{1+e\cos(\nu)}{1-e^2}
     \cos(2g+2\nu)
   \right] d\nu
\end{multline}
which yields:
\begin{multline}
 S_1(L'',G'',H'',l,g)=
 \frac{\mu^2R^2}{4{G''}^5}
 \left[
 ({G''}^2-3{H''}^2)(l-\nu-e\sin(\nu))
 \right.+
%  \\
%  \frac{\mu^2R^2}{4{L''}^3{G''}^2}
%  \left[
%     (3{H''}^2-{G''}^2)(\nu+e\sin(\nu))
%  \right.+
 \\
 \left. 
    ({G''}^2-{H''}^2)
    \left(
       \frac{3}{2}\sin(2g+2\nu)+
       \frac{3e}{2}\sin(2g+\nu)+
       \frac{1}{2}\sin(2g+3\nu)
    \right)
 \right]
\end{multline}

In the same way from \eqref{condition-C-S-2} we get:

\begin{equation}
 \KK_2(L'',G'',H'')=
 \secl{
 \underbrace{
 \frac{1}{2}\frac{\partial^2\HH_0}{\partial L^2}
 \left(
   \frac{\partial S_1}{\partial l}
 \right)^2
   +
\frac{\partial\HH_1}{\partial L}\frac{\partial S_1}{\partial l} 
   +
\frac{\partial\HH_1}{\partial G}\frac{\partial S_1}{\partial g} 
%    +
% \frac{\partial\HH_1}{\partial H}\frac{S_1}{\partial h} 
 }_{=:\bar\HH(L'',G'',H'',l,g)}
 }
\end{equation}
(every term on the R.H.S. is computed in $(L'',G'',H'',l,h,g)$) 
% \[
%  \secl{
%  \frac{3\mu^2}{{L''}^4}
%  \left(
%    \frac{\partial S_1}{\partial l}
%  \right)^2
%    +
% \frac{\partial\HH_1}{\partial L}\frac{\partial S_1}{\partial l} 
%    +
% \frac{\partial\HH_1}{\partial G}\frac{\partial S_1}{\partial g} 
%  }
% \]

and from \eqref{condition-C-P-2}:
\begin{equation}
 \frac{\partial S_2}{\partial l}(L'',G'',H'',l,g)=
 -
 \frac{2\mu^2R^2}{{L''}^3}
 \perd{
 \frac{1}{2}
 \frac{\partial^2\HH_0}{\partial L^2}
 \left(
   \frac{\partial S_1}{\partial l}
 \right)^2
   +
\frac{\partial\HH_1}{\partial L}\frac{S_1}{\partial l} 
   +
\frac{\partial\HH_1}{\partial G}\frac{S_1}{\partial g} 
%    +
% \frac{\partial\HH_1}{\partial H}\frac{S_1}{\partial h} 
 }
\end{equation}

% 
% \begin{multline*}
%  \KK(L'',G'',H'')=
%  \left[
%  \frac{3\mu^2}{L^4}
%  \left(
%  \frac{2\mu^2R^2}{{L''}^3}
%    \left[
%      \left(
%        -\frac{1}{2}+\frac{3}{2}\frac{{H''}^2}{{G''}^2}
%      \right)\frac{\aaa^3}{\rr^3}
%      \right.\right.\right.
%      +
%      \\
%      \left.\left.\left.
%      \left(
%        \frac{3}{2}-\frac{3}{2}\frac{{H''}^2}{{G''}^2}
%      \right)\frac{\aaa^3}{\rr^3}\cos(2g+2\nu)
%    \right]
%  \right)^2+
%  \right]_s=
% \end{multline*}
% 
% 
% \[
%  \frac{\mu^4R^2}{2L^6}
%    \left[
%      \left(
%        -\frac{1}{2}+\frac{3}{2}\frac{H^2}{G^2}
%      \right)\frac{\aaa^3}{\rr^3}
%      +
%      \left(
%        \frac{3}{2}-\frac{3}{2}\frac{H^2}{G^2}
%      \right)\frac{\aaa^3}{\rr^3}\cos(2g+2\nu)
%    \right]
% \]

If we compute the term $\bar\HH$ on the right hand side of \eqref{condition-C-S-2}, we get:
\[
 \bar\HH=\bar\HH_0+
 \bar\HH_3\left(\frac{\aaa}{\rr}\right)^3+
 \bar\HH_6\left(\frac{\aaa}{\rr}\right)^3
\]
with:
\begin{align*}
 \bar\HH_0(L,G,H)=&
 \frac{3 \mu ^6 R^4 \left(G^2-3 H^2\right)^2}{32 G^{10} L^4}
 \\
 \bar\HH_3\nu,g)=&
 -\frac{3 \mu ^6 R^4}{16 G^8 L^7}
 \left(-12 e G^2 H^2 L \sin ^2(g+\nu) \cos (2 g+\nu)
 \right. 
 \\
 &\left. 
 -4 e G^2 H^2 L \sin ^2(g+\nu) \cos (2 g+3 \nu)+
 \right. 
 \\
 &\left. 
 12 e H^4 L \sin ^2(g+\nu) \cos (2 g+\nu)+
 \right. 
 \\
 &\left. 
 4 e H^4 L \sin ^2(g+\nu) \cos (2 g+3 \nu)+
 \right. 
 \\
 &\left. 
 3 G^5 \cos (2 (g+\nu))
 -12 G^3 H^2 \cos (2 (g+\nu))
 \right. 
 \\
 &\left. 
 -12 G^2 H^2 L \sin ^2(g+\nu) \cos (2 (g+\nu))+
 \right. 
 \\
 &\left. 
 9 G H^4 \cos (2 (g+\nu))+
%  \right. 
%  \\
%  &\left. 
 12 H^4 L \sin ^2(g+\nu) \cos (2 (g+\nu))
 \right. 
 \\
 &\left. 
 -G^5+6 G^3 H^2-9 G H^4\right)
 \\
 \bar\HH_6(L,G,H,\nu,g)=&
 -\frac{9 \mu ^6 R^4}{32 G^4 L^{10}}
 \left(9 G^4 \cos ^2(2 (g+\nu))-6 G^4 \cos (2 (g+\nu))
 \right. 
 \\
 &\left. 
 -18 G^2 H^2 \cos ^2(2 (g+\nu))+
 24 G^2 H^2 \cos (2 (g+\nu))+
 \right. 
 \\
 &\left. 
 9 H^4 \cos ^2(2 (g+\nu))
 -18 H^4 \cos (2 (g+\nu))+
 \right. 
 \\
 &\left. 
 G^4-6 G^2 H^2+9 H^4\right)
\end{align*}
Using \eqref{eqn:change-of-variable-l-nu} we can take the average
of $\bar\HH$ with  respect to $l$ and $g$ (by passing to the integral in $\nu$). We hence get:
\begin{multline*}
 \KK_2(L'',G'',H'')=
 -\frac{3 \mu ^6 R^4}{512 {G''}^{13} {L''}^5}
 \left(99 {G''}^8-48 {G''}^7 {L''}
 \right.
 \\
 \left.
 -2 {G''}^6 \left(167 {H''}^2+
 495 {L''}^2\right)+
 288 {G''}^5 {H''}^2 {L''}+
 \right.
 \\
 \left.
 {G''}^4 
 \left(307 {H''}^4+2860 {H''}^2 {L''}^2+1155 {L''}^4\right)
 \right.
 \\
 \left.
 -432 {G''}^3 {H''}^4 {L''}-70 {G''}^2 \left(37 {H''}^4 {L''}^2+45 {H''}^2 {L''}^4\right)+2835 {H''}^4 {L''}^4\right)
\end{multline*}
All the computations above have been checked in the Mathematica notebook avaliable at
https://www.wolframcloud.com/obj/claudiosaccon/VonZeipel2.nb\ . The notations inside the notebook
should be self--explanatory: 
\begin{gather*}
 \texttt{H0}=\HH_0,\quad
 \texttt{H1}=\HH_1,\quad
 \texttt{H1sec}=\HH_{1,sec},\quad
 \texttt{K1}=\KK_1,\quad
 \texttt{K2}=\KK_2,
 \\
 \texttt{S1}=S_1,\quad
 \texttt{DS1l}=\frac{\partial S_1}{\partial l},\quad
 \texttt{DS1g}=\frac{\partial S_1}{\partial g},\quad
 \texttt{S2}=S_2,
\end{gather*}
while the terms  \texttt{RH*} are sum up to form $\HH_2$. The rows of the notebook should be evaluated
in sequence to provide the required expressions of
$S=S_0+S_1+S_2$ and the corresponding derivatives. 
As said before  also $S_2$ is computed according to
our formulas. Since $S_2$  is made up by a very long list of terms, we do not insert it here. 
Notice that
\cite{Brouwer1959} considers
$S_2$ as an upper order quantity  and  therefore neglectes it (making no attempts to compute it).

\appendix

\section{The Hamiltonian approach. Canonical transformations and generating functions}

Let
$\HH:\real^{N}\times\real^{N}\to \real$ be a smooth function (the \emph{Hamiltonian}), we
we say that $p,q:\real\to\real^N$ are solutions to the Hamiltonian equations, if: 
\begin{equation}\label{eqn:hamilton-formulation}
 \left\{
   \begin{aligned}
    &\dot p
    %\left(=\frac{\partial}{\partial t}p \right)
    =
    -\frac{\partial}{\partial q}\HH(p,q)^t,
    \\
    &\dot q
    %\left(=\frac{\partial}{\partial t}q\right)
    =
    \frac{\partial}{\partial p}\HH(p,q)^t
   \end{aligned}
 \right.
\end{equation}
As well known \eqref{eqn:hamilton-formulation} describe the dynamics of a system only
effected by conservative forces,
according  the laws of Classical Mechanics.  
The variables $q$ are called ``coordinates'' while $p$ are the ``momenta''.
 of the system.
 
In case we need to consider the presence of dissipative forces we
add two two terms to \eqref{eqn:hamilton-formulation}:
\begin{equation}\label{eqn:hamilton-with-external-force}
 \left\{
   \begin{aligned}
    &\dot p
    %\left(=\frac{\partial}{\partial t}p \right)
    =
    -\frac{\partial}{\partial q}\HH(p,q)^t
    +f_1(p,q),
    \\
    &\dot q
    %\left(=\frac{\partial}{\partial t}q\right)
    =
    \frac{\partial}{\partial p}\HH(p,q)^t
    +f_2(p,q).
   \end{aligned}
 \right.
\end{equation}
where $f_1,f_2:\real^N\times\real^{N}\to\real^N$. It will be sometimes convenient to indicate 
\(
 v=\begin{pmatrix}
     p\\q                    
 \end{pmatrix}
 \)
 and rewrite
\eqref{eqn:hamilton-with-external-force} as:
\begin{equation}\label{eqn:hamiltonian-concise}
 \dot v=J\nabla \HH(v)+f,
\end{equation}
where:
\[
 %\nabla \HH(v)=
 \nabla \HH(p,q)=
 \begin{pmatrix}
  \dfrac{\partial}{\partial p}\HH(p,q)
  \\
  \ 
  \\
  \dfrac{\partial}{\partial q}\HH(p,q)
 \end{pmatrix}
%  =
%  \frac{\partial \HH(v)}{\partial v}^t
 \ ,\ 
 J:=
 \begin{pmatrix}
  0&-I_N
  \\
  \\
  I_N&0
 \end{pmatrix}
  \ ,\ 
  f(v)=f(p,q):=
 \begin{pmatrix}
  f_1(p,q)
  \\
  \ 
  \\
  f_2(p,q)
 \end{pmatrix}
\]
($I_N$ denotes the $N$-dimensional identity matrix). 
Notice that $J^2=-I_N$.

\begin{dfn}
 Let $\Phi:\real^{2N}\to\real^{2N}$ be a diffeomorphism i.e. a bijection such that $\Phi$ and $\Phi^{-1}$ are smooth. We say that $\Phi$ is a \emph{canonical transformation},  if for any solution $v$ of 
 \(
 \dot v=J\nabla \HH(v)
 \),  
 the transformed function $V:=\Phi(v)$ is a solution of
 \(
  \dot V=J\nabla \KK(V),
 \)
where $\KK(V):=\HH(\Phi^{-1}(V))$. This is usually expressed by saying that ``$\Phi$ preserves the form of Hamilton's equations''.
\end{dfn}

\begin{thm}
 Let $\Phi$ be a diffeomorphism. 
 If $\dfrac{\partial\Phi}{\partial v}$ is symplectic, that is if:  
 \begin{equation}\label{eqn:jacobian-symplectic}
  \frac{\partial\Phi}{\partial v}
  J
  \frac{\partial\Phi}{\partial v}^t
  =
  J
 \end{equation}
 holds,
 then $\Phi$ is a canonical transformation.
\end{thm}
\begin{proof}
 From the definition of $\KK$ we have: $\KK(\Phi(v))=\HH(v)$.
 Then:
 \[
  \frac{\partial \HH(v)}{\partial v}=
  \frac{\partial \KK(V)}{\partial V}
  \frac{\partial\Phi(v)}{\partial v}.
 \]
% 
% 
% 
% 
%  From $\KK(V):=\HH(\Phi^{-1}(V))$, we get:
%  \[
%   \frac{\partial \KK(V)}{\partial V}=
%   \frac{\partial \HH(v)}{\partial v}
%   \frac{\partial\Phi^{-1}(V)}{\partial V}=
%   \frac{\partial \HH(v)}{\partial v}
%   \left(\frac{\partial\Phi(v)}{\partial v}\right)^{-1}
%   \Leftrightarrow
%   \frac{\partial \HH(v)}{\partial v}=
%   \frac{\partial \KK(V)}{\partial V}
%   \frac{\partial\Phi(v)}{\partial v}.
%  \]
 Now, if
 \(
 \dot v=J\nabla \HH(v)
 \)
 and
 $V=\Phi(v)$, we have:
 \begin{equation}\label{eqn:equation-Phi(v)}
  \dot V=\frac{\partial\Phi(v)}{\partial v}\dot v=
  \frac{\partial\Phi(v)}{\partial v}
  J
  \frac{\partial \HH(v)}{\partial v}^t=
  \frac{\partial\Phi(v)}{\partial v}
  J
  \frac{\partial\Phi(v)}{\partial v}^t
  \frac{\partial \KK(V)}{\partial V}^t.  
 \end{equation}

 We can conclude that, if \eqref{eqn:jacobian-symplectic} holds,
 then
 \(
  \dot V=J\nabla \KK(V).
 \)
%  
%  (???VICEVERSA)
%  To prove the reverse implication take any $v_0\in\real^{2N}$ and let $V_0=\Phi(v_0)$. Define $v(t)$ as a the unique solution to 
%  $\dot v=J\nabla \HH(v)$, with $v(0)=v_0$. We know that 
%  $V(t):=\Phi(v(t))$ solves $\dot V=J\nabla \KK(V)$ with $V(0)=V_0$.
%  Using \eqref{eqn:equation-Phi(v)} we get:
%  \[
%   J\frac{\partial \KK(V(t))}{\partial V}^t= 
%   \dot V(t)=
%   \frac{\partial\Phi(v(t))}{\partial v}
%   J
%   \frac{\partial\Phi(v(t))}{\partial v}^t
%   \frac{\partial \KK(V(t))}{\partial V}^t.
%  \]
% Taking $t=0$ we get 
% \(
% \displaystyle{
% J\nabla \HH(v_0)=
%   \dfrac{\partial\Phi(v_0)}{\partial v}
%   J
%   \dfrac{\partial\Phi(v_0)}{\partial v}^t
%   \nabla \HH(v_0).
% }
% \) 
% for all $v_0\in\real^{2N}$.
% 
% QUESTO PER\`O non sembra implicare 
% \(
%   J=
%   \dfrac{\partial\Phi(v_0)}{\partial v}
%   J
%   \dfrac{\partial\Phi(v_0)}{\partial v}^t
% \).
\end{proof}

\begin{rmk}
 As shown in \thetag{1} of \eqref{prop:properties-symplectic-matrices} the
 equality
 \eqref{eqn:jacobian-symplectic} is equivalent to:
%  \(
%  V=\begin{pmatrix}
%      P\\Q                    
%  \end{pmatrix}
%  \):

 \begin{equation}\label{eqn:properties-symplectic-matrices-subblocks}
  \frac{\partial P}{\partial p} \frac{\partial P}{\partial q}^t 
  =
  \frac{\partial P}{\partial q}\frac{\partial P}{\partial p}^t
  \qquad
  \frac{\partial Q}{\partial p} \frac{\partial Q}{\partial q}^t 
  =
  \frac{\partial Q}{\partial q}\frac{\partial Q}{\partial p}^t
  \qquad
  \frac{\partial P}{\partial p}\frac{\partial Q}{\partial q}^t
  -
  \frac{\partial P}{\partial q}\frac{\partial Q}{\partial p}^t
  =I_N
 \end{equation}
 or equivalently to:
  \begin{equation}\label{eqn:properties-symplectic-matrices-subblocks-T}
  \frac{\partial P}{\partial p}^t \frac{\partial Q}{\partial p} 
  =
  \frac{\partial Q}{\partial p}^t\frac{\partial P}{\partial p}
  \qquad
  \frac{\partial P}{\partial q}^t \frac{\partial Q}{\partial q} 
  =
  \frac{\partial Q}{\partial q}^t\frac{\partial P}{\partial q}
  \qquad
  \frac{\partial P}{\partial p}^t\frac{\partial Q}{\partial q}
  -
  \frac{\partial Q}{\partial p}^t\frac{\partial P}{\partial q}
  =I_N
 \end{equation}
 
%  \begin{align}
%   \frac{\partial P}{\partial p} \frac{\partial P}{\partial q}^t 
%   &=
%   \frac{\partial P}{\partial q}\frac{\partial P}{\partial p}^t
%   \label{eqn:jacobian-symplectic-submatrices-1}
%   \\
%   \frac{\partial Q}{\partial p} \frac{\partial Q}{\partial q}^t 
%   &=
%   \frac{\partial Q}{\partial q}\frac{\partial Q}{\partial p}^t
%   \label{eqn:jacobian-symplectic-submatrices-2}
%   \\
%   \frac{\partial P}{\partial p}\frac{\partial Q}{\partial q}^t
%   &-
%   \frac{\partial P}{\partial q}\frac{\partial Q}{\partial p}^t
%   =I_N.
%   \label{eqn:jacobian-symplectic-submatrices-3}
%  \end{align}
\end{rmk}

\begin{rmk}
 If $\dfrac{\partial\Phi}{\partial v}$ is symplectic, then the 
 change of variables $V=\Phi(v)$ transforms the
 nonconservative equation \eqref{eqn:hamilton-with-external-force}
 into:
 \[
  \dot V=J\nabla \KK(V)+F(V)
 \]
 where $\KK=\HH\circ\Phi^{-1}$ and:
 \begin{equation*}
  F(V)=\frac{\partial\Phi}{\partial v}(\Phi^{-1}(V))f(\Phi^{-1}(V))
  =
  \left(\frac{\partial\Phi^{-1}}{\partial V}(V)\right)^{-1}f(\Phi^{-1}(V)).
  \end{equation*}
  But since $\dfrac{\partial\Phi}{\partial v}$ is symplectic, so is
  $\dfrac{\partial\Phi^{-1}}{\partial V}$ (see section \ref{sect:symplectic} and we have:
  \[
   \frac{\partial\Phi^{-1}(V)}{\partial V}=
   \begin{pmatrix}
    \dfrac{\partial p(P,Q)}{\partial P}
    &\dfrac{\partial p(P,Q)}{\partial Q}
    \\
    \dfrac{\partial q(P,Q)}{\partial P}
    &\dfrac{\partial q(P,Q)}{\partial Q}
   \end{pmatrix}^{-1}=
   \begin{pmatrix}
    \dfrac{\partial q(P,Q)}{\partial Q}^t
    &-\dfrac{\partial p(P,Q)}{\partial Q}^t
    \\
    -\dfrac{\partial q(P,Q)}{\partial P}^t
    &\dfrac{\partial p(P,Q)}{\partial P}^t.
   \end{pmatrix}
  \]
  If $\Phi^{-1}(P,Q)=\begin{pmatrix}p(P,Q)\\q(P,Q)\end{pmatrix}$,
  we have finally:
  \begin{align}\label{formula-transformation-nonconservative-force}
   F_1(P,Q)=&
   \dfrac{\partial q(P,Q)}{\partial Q}^tf_1(\Phi^{-1}(P,Q))-
   \dfrac{\partial p(P,Q)}{\partial Q}^tf_2(\Phi^{-1}(P,Q))
   \\
   F_2(P,Q)=&-
   \dfrac{\partial q(P,Q)}{\partial P}^tf_1(\Phi^{-1}(P,Q))+
   \dfrac{\partial p(P,Q)}{\partial P}^tf_2(\Phi^{-1}(P,Q))
  \end{align}

\end{rmk}

A possible way of obtaining a canonical transformation is
via a generating function.
\begin{dfn}[generating functions]
 Let $S:\real^N\times\real^N\to\real$ be a smooth function.
 We see $S$ as $S(P,q)$ (a combination of ``new variables''
 $P$ and ``old ones'' $q$).
 It is simple to check that the equations:
  \begin{equation}\label{eqn:definition-canonical}
   p=\frac{\partial S}{\partial q}(P,q)
   \quad,\quad
   Q=\frac{\partial S}{\partial P}(P,q)
  \end{equation}
 implicitely define a canonical transformation 
 $(P,Q)=\Phi(p,q)$, in the sense that 
 $\Phi:\real^N\times\real^N\to\real^N\times\real^N$ is canonical and:
 \[
  (P,Q)=\Phi(p,q)
  \Leftrightarrow
  \mbox{ \eqref{eqn:definition-canonical} holds }
  \qquad\qquad
  \forall\ p,q,P,Q.  
 \]
% 
%  
%  $(P,Q)=\Phi(p,q)$ if and only if \eqref{eqn:definition-canonical} holds.
 
 Such an $S$ will be called a 
 \emph{generating function}
 for the canonical transformation $\Phi$.

\end{dfn}

\begin{thm}\label{thm:generating-function}
  Let
  \(
   \Phi(p,q)=
   \begin{pmatrix}
    P(p,q)
    \\
    Q(p,q)
   \end{pmatrix}
  \)
  be a diffeomorphism and 
  $\dfrac{\partial\Phi}{\partial v}$
  be symplectic.
  
  Let $(p_0,q_0)$ be a point such that 
  $\det\left(\dfrac{\partial\Phi_1}{\partial p}(p_0,q_0)\right)\neq0$
  and let $(P_0,Q_0)=\Phi(p_0,q_0)$.
  Then there exist a neighboorhood $W$ of $(P_0,q_0)$,
  a neighboorhood $W_1$ of $(p_0,Q_0)$, and  a
  smooth function $S:W\to\real$ such that
  $S$ is a generating function for $\Phi$ in $W$.

\end{thm}
\begin{proof}
 Let $G:\real^{N}\times\real^N\times\real^N\times\real^N$ be
 defined by:
 \[
  G(P,Q,p,q):=
  \begin{pmatrix}
   P
   \\
   Q
  \end{pmatrix}
  -\Phi(p,q)=
  \begin{pmatrix}
   P
   \\
   Q
  \end{pmatrix}
  -
  \begin{pmatrix}
   \Phi_1(p,q)
   \\
   \Phi_2(p,q)
  \end{pmatrix}.
 \]
 We have:
 \[
  \frac{\partial G}{\partial(P,Q,p,q)}=
  \begin{pmatrix}
   I_N&0
   &-\dfrac{\partial\Phi_1}{\partial p}
   &-\dfrac{\partial\Phi_1}{\partial q}
   \\
   0&I_N\
   &-\dfrac{\partial\Phi_2}{\partial p}
   &-\dfrac{\partial\Phi_2}{\partial q}
  \end{pmatrix}.
 \]
 From the assumption we derive that:
 \[
  \det 
  \left(
   \dfrac{\partial G}{\partial(Q,p)}
  \right)=
  \det
  \left(
   \begin{pmatrix}
   0
   &-\dfrac{\partial\Phi_1}{\partial p}
   \\
   I_N
   &-\dfrac{\partial\Phi_2}{\partial p}
  \end{pmatrix}
  \right)=
  \det\left(
  \frac{\partial\Phi_1}{\partial p}
  \right)
  \neq0
 \]
 By the implicit function Theorem there exist
 a neighboorhood $W$ of $(P_0,q_0)$, a neighboorhood $W_1$
 of $(Q_0,p_0$ and a map
 $\Psi:W\to\real^N\times\real^N$ such that:
  \[
   (P,Q)=\Phi(p,q)
   \quad\Leftrightarrow\quad
   (Q,p)=\Psi(P,q).
  \]
  Moreover:
  \begin{multline*}
   \frac{\partial\Psi}{\partial(P,q)}=
   -
   \begin{pmatrix}
   0
   &-\dfrac{\partial\Phi_1}{\partial p}
   \\
   I_N
   &-\dfrac{\partial\Phi_2}{\partial p}
  \end{pmatrix}^{-1}
  \begin{pmatrix}
   I_N
   &-\dfrac{\partial\Phi_1}{\partial q}
   \\
   0
   &-\dfrac{\partial\Phi_2}{\partial q}
  \end{pmatrix}=
   \\
   \begin{pmatrix}
   \dfrac{\partial\Phi_2}{\partial p}\left(\dfrac{\partial\Phi_1}{\partial p}\right)^{-1}
   &-I_N
   \\
   \left(\dfrac{\partial\Phi_1}{\partial p}\right)^{-1}
   &0
  \end{pmatrix}
  \begin{pmatrix}
   I_N
   &-\dfrac{\partial\Phi_1}{\partial q}
   \\
   0
   &-\dfrac{\partial\Phi_2}{\partial q}
  \end{pmatrix}=
  \\
   \begin{pmatrix}
   \dfrac{\partial\Phi_2}{\partial p}\left(\dfrac{\partial\Phi_1}{\partial p}\right)^{-1}
   &
   -\dfrac{\partial\Phi_2}{\partial p}\left(\dfrac{\partial\Phi_1}
   {\partial p}\right)^{-1}
   \dfrac{\partial\Phi_1}{\partial q}
   +\dfrac{\partial\Phi_2}{\partial q}
  \\
  \left(\dfrac{\partial\Phi_1}{\partial p}\right)^{-1}
  &
  \left(\dfrac{\partial\Phi_1}{\partial p}\right)^{-1}
  \dfrac{\partial\Phi_1}{\partial q}
   \end{pmatrix}.
  \end{multline*}
 We claim that the above matrix is symmetric. The upper left
 block is symmetric since:
 \[
  \dfrac{\partial\Phi_2}{\partial p}
  \left(\dfrac{\partial\Phi_1}{\partial p}\right)^{-1}=
  \left(\dfrac{\partial\Phi_1}{\partial p}\right)^{-t}
  \left(\dfrac{\partial\Phi_2}{\partial p}\right)^t
  \Leftrightarrow
  \left(\dfrac{\partial\Phi_1}{\partial p}\right)^{t}
  \dfrac{\partial\Phi_2}{\partial q}=
  \left(\dfrac{\partial\Phi_2}{\partial p}\right)^t
  \dfrac{\partial\Phi_1}{\partial p}
 \]
 and the latter is exactly the first equality in \eqref{eqn:properties-symplectic-matrices-subblocks-T}. 
 In the same way the lower right block is symmetric:
%  $\dfrac{\partial\Phi_2}{\partial q}
%   \left(\dfrac{\partial\Phi_1}{\partial p}\right)^{-1}=
%   \left(\dfrac{\partial\Phi_1}{\partial p}\right)^{-t}
%   \dfrac{\partial\Phi_2}{\partial q}^t$ is symmetric
  
   \[
  \left(\dfrac{\partial\Phi_1}{\partial p}\right)^{-1}
  \dfrac{\partial\Phi_1}{\partial q}=
  \left(\dfrac{\partial\Phi_1}{\partial q}\right)^{t}
  \left(\dfrac{\partial\Phi_1}{\partial p}\right)^{-t}
  \Leftrightarrow
  \dfrac{\partial\Phi_1}{\partial q}
  \left(\dfrac{\partial\Phi_1}{\partial p}\right)^t=
  \dfrac{\partial\Phi_1}{\partial p}
  \left(\dfrac{\partial\Phi_1}{\partial q}\right)^t
 \]
 which is  the first equality in \eqref{eqn:properties-symplectic-matrices-subblocks}. Finally:
 \begin{multline*}
  \left(\dfrac{\partial\Phi_1}{\partial p}\right)^{-t}=
  -\dfrac{\partial\Phi_2}{\partial p}\left(\dfrac{\partial\Phi_1}
   {\partial p}\right)^{-1}
   \dfrac{\partial\Phi_1}{\partial q}
   +\dfrac{\partial\Phi_2}{\partial q}
   \Leftrightarrow
   \\
   I_N=
   -\dfrac{\partial\Phi_2}{\partial p}\left(\dfrac{\partial\Phi_1}
   {\partial p}\right)^{-1}
   \underbrace{\dfrac{\partial\Phi_1}{\partial q}
   \left(\dfrac{\partial\Phi_1}{\partial p}\right)^{t}}_{
   =\frac{\partial\Phi_1}{\partial p}
   \left(\frac{\partial\Phi_1}{\partial q}\right)^{t}
   }
   +\dfrac{\partial\Phi_2}{\partial q}
   \left(\dfrac{\partial\Phi_1}{\partial p}\right)^{t}
   \Leftrightarrow
   \\
   I_N=
   -\dfrac{\partial\Phi_2}{\partial p}
   \left(\dfrac{\partial\Phi_1}{\partial q}\right)^t
   +\dfrac{\partial\Phi_2}{\partial q}
   \left(\dfrac{\partial\Phi_1}{\partial p}\right)^{t}
   \Leftrightarrow
   \\
   I_N=(I_N^t=)
   -\dfrac{\partial\Phi_1}{\partial q}
   \left(\dfrac{\partial\Phi_2}{\partial p}\right)^t
   +\dfrac{\partial\Phi_1}{\partial p}
   \left(\dfrac{\partial\Phi_2}{\partial q}\right)^{t}
 \end{multline*}
  and the last equality coincides with the third one in
  \eqref{eqn:properties-symplectic-matrices-subblocks}.
  
  Since $\dfrac{\partial\Psi}{\partial(P,q)}$ is symmetric, there
  exists a ``potential'' $S:\to\real$ such that:
  \[
   \Psi(P,q)=\nabla S(P,q)\Leftrightarrow
   \Psi_1(P,q)=\frac{\partial S}{\partial P}(P,q)
   \ ,\ 
   \Psi_2(P,q)=\frac{\partial S}{\partial q}(P,q)
  \]
  for all $(P,q)$ in $W$. This concludes the proof.
\end{proof}

% 
% \begin{rmk}
%  Other forms of generating functions are possible.
%  
%  XXXX
%  
%  XXXX
%  
% \end{rmk}

\section{Simplectic matrices}\label{sect:symplectic}

\begin{dfn}
 A $2N\times2N$ matrix $M$ is said to be \emph{symplectic}
 if $MJM^t=M$.
\end{dfn}

\begin{prop}[Properties of symplectic matrices]\label{prop:properties-symplectic-matrices}
 Let $M,M_1$ be $2N\times2N$ matrices and suppose that
 \(
 \displaystyle{
 M=
 \begin{pmatrix}
  A&B
  \\
  C&D
 \end{pmatrix}
 }
 \), where $A,B,C,D$ are four $N\times N$ matrices.
 \begin{enumerate}
  \item 
  $M$ is symplectic if and only if:
  \[
   AB^t=BA^t
   \quad,\quad
    CD^t=DC^t
   \quad,\quad
   AD^t-BC^t=I_N
  \]
  if and only if:
   \[
   A^tC=C^tA
   \quad,\quad
    B^tD=D^tB
   \quad,\quad
   A^tD-C^tB=I_N
  \]
 \item
 If $M$ and $M_1$ are symplectic, then $MM_1$ is symplectic.
 \item
 If $M$ is symplectic, than $M$ is invertible and 
 \[
  M^{-1}=-JM^tJ=
  \begin{pmatrix}
  D^t&-B^t
  \\
  -C^t&A^t
  \end{pmatrix}
  .
 \]
 \item
 If $M$ is symplectic, than $M^t$ is symplectic.
 \end{enumerate}
\end{prop}
\begin{proof}
 \begin{enumerate}
  \item 
  We have:
  \begin{multline*}
   MJM^t=
 \begin{pmatrix}
  A&B
  \\
  C&D
 \end{pmatrix}
 \begin{pmatrix}
  0&-I_N
  \\
  I_N&0
 \end{pmatrix}
 \begin{pmatrix}
  A^t&C^t
  \\
  B^t&D^t
  \end{pmatrix}
   =
   \\
  \begin{pmatrix}
  B&-A
  \\
  D&-C
 \end{pmatrix}
 \begin{pmatrix}
  A^t&C^t
  \\
  B^t&D^t
  \end{pmatrix}=
  \begin{pmatrix}
   BA^t-AB^t&BC^t-AD^t
  \\
  DA^t-CB^t-&DC^t-CD^t.
  \end{pmatrix}
  \end{multline*}
  So the condition $MJM^t=J$ turns out to be equivalent to 
  the first set of equalities above (the possible fourth equality is easily derived from the third one, by taking the transpose).
  The second set of equalities follows from the fact that $M^t$
  is symplectic, as shown in \thetag{4}.
  \item
  We have:
  \[
   MM_1J(MM_1)^t=M\underbrace{M_1JM_1^t}_{=J}M^t=MJM^t=J.
  \]
  \item
  We have:
  \[
   M(-JM^tJ)=-MJM^tJ=-JJ=I_N.
  \]
  This implies that $-JM^tJ$ is a left inverse. A standard argument
  shows that $M$ is invertible and $M^{-1}=-JM^tJ$.
  \item
  From the previous statement we have $-JM^tJM=I_N$. By multiplying
  by $J$, we get that $M^t$ is symplectic.
 \end{enumerate}

\end{proof}

\bibliography{Niccolo}
\bibliographystyle{plain}
\end{document}